\documentclass[a4paper,12pt,twoside]{article}
\usepackage{latexsym,amsmath,amssymb}\usepackage{amsthm,amscd,mathrsfs}
\usepackage{bm}
\usepackage[dvipdfmx]{graphicx}
\usepackage{fancybox}
\usepackage{multicol}
\usepackage{tabularx}
\usepackage[top=3.5cm,bottom=3.5cm,left=3cm,right=3cm]{geometry}
\usepackage{graphics}
\usepackage{tikz}
\usepackage{comment}

\newtheorem{thm}{Theorem}[section]
\newtheorem{prop}{Proposition}[section]
\newtheorem{cor}{Corollary}[section]
\newtheorem{lem}{Lemma}[section]

\newtheorem{df}{Definition}[section]

\theoremstyle{definition}

\newtheorem{rem}{Remark}[section]

\theoremstyle{theorem}

\makeatletter
 
 \@addtoreset{equation}{section}
\makeatother

\title{Some Inclusion Relations for Fuzzy Sets}
\author{N. SOMEYAMA$^*$}
\date{{\small $^*$Shin-yo-ji Buddhist Temple, 5-44-4 Minamisenju, Arakawa-ku, Tokyo 116-0003 Japan}\\
{\small E-mail: {\tt philomatics@outlook.jp}}\\
{\small ORCID iD: https://orcid.org/0000-0001-7579-5352}}
\begin{document}
\maketitle
\markboth{N. Someyama}{Some Inclusion Relations for Fuzzy Sets}

\begin{abstract}
We give some inclusion relations for arbitrary fuzzy sets with reference to famous inequalities.
In particular, we can know that the bounded sum and the algebraic product go well together.
We would like to propose the concept of `Fuzzy Set Inequalities' through the present note.
\end{abstract}
\vspace{2mm}



\section{Introduction}
The concept of fuzzy sets \cite{Z} was introduced by L.A. Zadeh(1921-2017) in 1965.
Fuzzy sets are extended sets that mathematically contain ambiguity.
The results of application in various disciplines are well known and there is no end to the list.
Since then, various operations have been introduced and applied to many studies.

In this section, we check definitions and symbols of operations for fuzzy sets so that there is no misunderstanding.

\subsection{Definition of fuzzy sets and the relations}
The usual set $S\subset X$ is called Crisp Set.
It is characterized by the defining function $\chi_A:X\to \{0,1\}$:
\begin{align*}
\chi_S(x):=
\begin{cases}
1, & x\in S; \vspace{1mm}\\
0, & x\notin S.
\end{cases}
\end{align*}

Fuzzy sets are defined with reference to that as follows.

\begin{df}
Let $X$ be a universal set.
$A$ is a fuzzy set on $X$ if and only if the chacteristic function of $A$ is defined as $\mu_A:X\to [0,1]$.
We write $\mathbb F(X)$ for the set of fuzzy sets on $X$.
In particular, we write $X$ (resp. $\emptyset$) for the fuzzy set whose membership function always takes value $1$ (resp. $0$).
\end{df}

The equality and inclusion relation for fuzzy sets are defined as follows.

\begin{df}
Let $A,B\in \mathbb F(X)$.
If $\mu_A(x)\le \mu_B(x)$ for all $x\in X$, we write $A\subseteq B$ and say that $A$ is included in $B$ or $B$ contains $A$.
The same applies to $A\supseteq B$.
In particular, if $\mu_A(x)=\mu_B(x)$ is satisfied for all $x\in X$, we write $A=B$ and say that $A$ and $B$ are equal.
\end{df}

Fuzzy sets are discussed via the membership functions, so properties of the membership functions will basically propagate to that of fuzzy sets.
Thus, are not inequalities that hold for mere numbers changed to inclusion relations and hold for fuzzy sets?
To investigate that is our aim in the present note.

We however check definitions of some operations for fuzzy sets, before getting into the main topic.

\subsection{Operations for fuzzy sets}

\begin{df}
Let $A,B\in \mathbb F(X)$.
We define the membership function of $A\cup B$ called `$A$ or $B$' by
\begin{align*}
\mu_{A\cup B}(x):=\max\{\mu_A(x),\ \mu_B(x)\}
\end{align*}
for $x\in X$.
Moreover, we define the membership function of $A\cap B$ called `$A$ and $B$' by
\begin{align*}
\mu_{A\cap B}(x):=\min\{\mu_A(x),\ \mu_B(x)\}
\end{align*}
for $x\in X$.
\end{df}

\begin{df}
Let $A,B\in \mathbb F(X)$.
We define the membership function of $A\dotplus B$ called Algebraic Sum of $A$ and $B$ by
\begin{align*}
\mu_{A\dotplus B}(x):=\mu_A(x)+\mu_B(x)-\mu_A(x)\mu_B(x)
\end{align*}
for $x\in X$.
Moreover, we define the membership function of $A\cdot B$ called Algebraic Product of $A$ and $B$ by
\begin{align*}
\mu_{A\cdot B}(x):=\mu_A(x)\mu_B(x)
\end{align*}
for $x\in X$.
In particular, we write $A^n$ for the product of multiplying $n$ bases $A$, i.e., $A\cdot A\cdot \cdots \cdot A$ ($n$ times).
\end{df}

We decide the set of natural numbers as $\mathbb N:=\{1,2,\ldots\}$ in the present note.

\begin{df}
Let $A,B\in \mathbb F(X)$.
We define the membership function of $A\oplus B$ called Bounded Sum of $A$ and $B$ by
\begin{align*}
\mu_{A\oplus B}(x):=\min\{\mu_A(x)+\mu_B(x),\ 1\}
\end{align*}
for $x\in X$.
Moreover, we define the membership function of $A\odot B$ called Bounded Product of $A$ and $B$ by
\begin{align*}
\mu_{A\odot B}(x):=\max\{\mu_A(x)+\mu_B(x)-1,\ 0\}
\end{align*}
for $x\in X$.
Furthermore, we define the membership function of $A\ominus B$ called Bounded Difference of $A$ and $B$ by
\begin{align*}
\mu_{A\ominus B}(x):=\max\{\mu_A(x)-\mu_B(x),\ 0\}
\end{align*}
for $x\in X$.
\end{df}

In addition to them, we introduce the division for fuzzy sets in the present note as follows.

\begin{df}
Let $A,B\in \mathbb F(X)$ be satisfied with $B\neq \emptyset$.
We define the membership function of $A\oslash B$ called {\rm Bounded Quotient} of $A$ and $B$ by
\begin{align*}
\mu_{A\oslash B}(x):=\min\left\{\frac{\mu_A(x)}{\mu_B(x)},\ 1\right\}
\end{align*}
for $x\in X$.
\end{df}

\begin{df}
Let $A\in \mathbb F(X)$ and $0\le \kappa\le 1$ be a real number.
We define the membership function of $\kappa A$ by
\begin{align*}
\mu_{\kappa A}(x):=\kappa \mu_A(x)
\end{align*}
for $x\in X$.
\end{df}

The following is a generalization of $A^n$, $n\in \mathbb N$.

\begin{df}
Let $A\in \mathbb F(X)$ and $p\ge 0$ be a real number.
We define the membership function of $A^p$ by
\begin{align*}
\mu_{A^p}(x):=\{\mu_A(x)\}^p
\end{align*}
for $x\in X$.
\end{df}

\begin{rem}
For any $A\in \mathbb F(X)$, one has $A^0=X$.
\end{rem}

\section{Main Results}
In what follows, we write $a$, $b$, $c$ and $d$ for $\mu_A(x)$, $\mu_B(x)$, $\mu_C(x)$ and $\mu_D(x)$ respectively, where $x\in X$ is fixed arbitrarily.

We begin with an easy result.
It is well known that 
\begin{align}
\label{eq:8abc}
(\alpha+\beta)(\beta+\gamma)(\gamma+\alpha)\ge 8\alpha\beta\gamma
\end{align}
for any $\alpha,\beta,\gamma\ge 0$.
The equality holds if and only if $a=b=c$.
We derive the corresponding inclusion relation.

\begin{thm}
\label{thm:8ABC}
Let $A,B,C\in \mathbb F(X)$ be satisfied with $0\le \mu_A(x)+\mu_B(x),\mu_B(x)+\mu_C(x),\mu_C(x)+\mu_A(x)\le 1$.
Then, one has
\begin{align}
\label{eq:8ABC}
\frac{A\oplus B}{2}\cdot \frac{B\oplus C}{2}\cdot \frac{C\oplus A}{2}\supseteq A\cdot B\cdot C.
\end{align}
The equality holds if and only if $A=B=C$.
\end{thm}

\begin{proof}
To see (\ref{eq:8ABC}), we consider membership functions.
Fix $x\in X$ arbitrarily.
We have, from the assumption `$0\le a+b,b+c,c+a\le 1$' and (\ref{eq:8abc}),
\begin{align*}
\mu_{\frac{A\oplus B}{2}\cdot \frac{B\oplus C}{2}\cdot \frac{C\oplus A}{2}}(x)&=\mu_{(A\oplus B)/2}(x)\mu_{(B\oplus C)/2}(x)\mu_{(C\oplus A)/2}(x) \\
&=\frac{1}{8}\mu_{A\oplus B}(x)\mu_{B\oplus C}(x)\mu_{C\oplus A}(x) \\
&=\frac{1}{8}\min\{a+b,1\}\min\{b+c,1\}\min\{c+a,1\} \\
&=\frac{1}{8}(a+b)(b+c)(c+a) \\
&\ge abc \\
&=\mu_{A\cdot B\cdot C}(x).
\end{align*}
Also, it is obvious from the above that `$A=B=C$' is the condition for the equal sign of (\ref{eq:8ABC}) to hold.
Hence, this completes the proof.
\end{proof}

\begin{rem}
Theorem \ref{thm:8ABC} is formulated as a fuzzy version of (\ref{eq:8abc}).
This is because considering the fuzzy set inequality of (\ref{eq:8abc})-type comes a risk that $\mu_{8(A\cdot B\cdot C)}(x)\ge 1$.
\end{rem}

We next see the following relational expressions in which distributive laws generally do not hold.

\begin{thm}
Let $A,B,C\in \mathbb F(X)$.
Then, we have
\begin{align}
A\cdot(B\oplus C)&\subseteq (A\cdot B)\oplus(A\cdot C), \label{eq:AcBoC}\\
(A\oplus B)\cdot C&\subseteq (A\cdot C)\oplus(B\cdot C). \label{eq:AoBcC}
\end{align}
The equality of (\ref{eq:AcBoC}) (resp. (\ref{eq:AoBcC})) holds if and only if $A=X$ (resp. $C=X$).
\end{thm}

\begin{proof}
We prove only (\ref{eq:AcBoC}) because the same applies to (\ref{eq:AoBcC}).
For that, we consider membership functions.
Fix $x\in X$ arbitrarily.
\begin{align*}
 \mu_{A\cdot(B\oplus C)}(x)&=\mu_A(x)\mu_{B\oplus C}(x) \\
 &=a\min\{b+c,\ 1\} \\
 &=\min\{ab+ac,\ a\} \\
 &\le \min\{\mu_{A\cdot B}(x)+\mu_{A\cdot C}(x),\ 1\} \\
 &=\mu_{(A\cdot B)\oplus(A\cdot C)}(x).
\end{align*}
Also, it is obvious from the above that `$A=X$' is the condition for the equal sign of (\ref{eq:AcBoC}) to hold.
Hence, this completes the proof.
\end{proof}

We next see the following inclusion relations, one for the bounded product and difference, the other for the bounded sum and difference.

\begin{thm}
Let $A,B,C\in \mathbb F(X)$. 
\begin{itemize}
\item[1)] If $A\subseteq C$, then
\begin{align}
\label{eq:AodBomCb}
A\odot(B\ominus C)\subseteq C\odot(B\ominus A),\quad (A\oplus B)\ominus C\subseteq (B\oplus C)\ominus A.
\end{align}
The equality holds if and only if $A=C$.
\item[2)] If $A\supseteq C$, then
\begin{align}
\label{eq:AodBomCp}
A\odot(B\ominus C)\supseteq C\odot(B\ominus A),\quad (A\oplus B)\ominus C\supseteq (B\oplus C)\ominus A.
\end{align}
The equality holds if and only if $A=C$.
\end{itemize}
\end{thm}

\begin{proof}
We prove only 1) because the same applies to 2).
For that, we consider membership functions.
Fix $x\in X$ arbitrarily.
Since $a\le c$ by the assumption: $A\subseteq C$, we have
\begin{align*}
\mu_{A\odot (B\ominus C)}(x)&=\max\{\mu_A(x)+\mu_{B\ominus C}(x)-1,\ 0\} \\
&=\max\{a+\max\{b-c,0\}-1,\ 0\} \\
&\le \max\{c+\max\{b-a,0\}-1,\ 0\} \\
&=\max\{\mu_C(x)+\mu_{B\ominus A}-1,\ 0\} \\
&=\mu_{C\odot(B\ominus A)}(x)
\end{align*}
and
\begin{align*}
\mu_{(A\oplus B)\ominus C}(x)&=\max\{\mu_{A\oplus B}(x)-\mu_C(x),\ 0\} \\
&=\max\{\min\{a+b,1\}-c,\ 0\} \\
&\le \max\{\min\{b+c,1\}-a,\ 0\} \\
&=\max\{\mu_{B\oplus C}-\mu_A(x),\ 0\} \\
&=\mu_{(B\oplus C)\ominus A}(x).
\end{align*}
Also, it is obvious from the above that `$A=C$' is the condition for the equal sign of (\ref{eq:AodBomCb}) to hold.
Hence, this completes the proof.
\end{proof}

Remark that $0\le a^p,b^p\le 1$ for any $p\ge 0$.

\begin{thm}
Let $A,B\in \mathbb F(X)$ and $p\ge 0$.
Then, one has
\begin{align}
\label{eq:ApcBpApdBp}
A^p\cup B^p\subseteq A^p\dotplus B^p.
\end{align}
The equality holds if and only if $A=X$ or $B=\emptyset$.
\end{thm}

\begin{proof}
To see (\ref{eq:ApcBpApdBp}), we consider membership functions.
First, 
\begin{align*}
\mu_{A^p\cup B^p}(x)=\max\{a^p,b^p\}
\end{align*}
for all $x\in X$.
Next, 
\begin{align*}
\mu_{A^p\dotplus B^p}(x)=a^p+b^p-a^pb^p
\end{align*}
for all $x\in X$.
We can now set $a\ge b$ without loss of generality.
Then, we have
\begin{align*}
\mu_{A^p\dotplus B^p}(x)-\mu_{A^p\cup B^p}(x)&=(a^p+b^p-a^pb^p)-a^p \\
&=b^p(1-a^p)\ge 0.
\end{align*}
Also, it is obvious from the above that `$A=X$ or $B=\emptyset$' is the condition for the equal sign of (\ref{eq:ApcBpApdBp}) to hold.
Hence, this completes the proof.
\end{proof}

There is an inequality called Rearrangement Inequality on indices of $\Sigma$-sums:
\begin{align*}
\sum_{i=1}^nx_iy_i\ge \sum_{j=1}^nx_jy_{n-j+1}
\end{align*}
if $x_1\ge x_2\ge \cdots \ge x_n$ and $y_1\ge y_2\ge \cdots \ge y_n$, in particular
\begin{align}
\label{eq:reaeq}
x_1y_1+x_2y_2\ge x_1y_2+x_2y_1.
\end{align}

\begin{thm}[{\bf Rearrangement inequality for fuzzy sets}]
Let $A,B,C,D\in \mathbb F(X)$ be satisfied with $A\supseteq C$ and $B\supseteq D$.
Then, one has
\begin{align}
\label{eq:ABCDADBC}
(A\cdot B)\oplus (C\cdot D)\supseteq (A\cdot D)\oplus (B\cdot C).
\end{align}
The equality holds if and only if $A=C$ or $B=D$.
\end{thm}

\begin{proof}
To see (\ref{eq:ABCDADBC}), we consider membership functions.
Fix $x\in X$ arbitrarily.
We have, from (\ref{eq:reaeq}),
\begin{align*}
\mu_{(A\cdot B)\oplus (C\cdot D)}(x)&=\min\{ab+cd,\ 1\} \\
&\ge \min\{ad+bc,\ 1\} \\
&=\mu_{(A\cdot D)\oplus (B\cdot C)}(x).
\end{align*}
Also, it is obvious from the above that `$A=C$ or $B=D$' is the condition for the equal sign of (\ref{eq:ABCDADBC}) to hold, since $(ab+cd)-(ad+bc)=(a-c)(b-d)$.
Hence, this completes the proof.
\end{proof}

Let us see if the triangle inequality holds for fuzzy sets as well.
Compare with the crisp case: $|\alpha-\beta|+|\beta-\gamma|\ge |\alpha-\gamma|$ for any $\alpha,\beta,\gamma\in \mathbb R$.

\begin{thm}[{\bf Triangle inequality for fuzzy sets}]
Let $A,B,C\in \mathbb F(X)$ that do {\bf not} satisfy $A\subseteq B\subseteq C$.
Then, one has
\begin{align}
\label{eq:TI}
(A\ominus B)\oplus (B\ominus C)\supseteq A\ominus C.
\end{align}
The equality holds if and only if $A=B$ or $B=C$.
Moreover, $A\subseteq B\subseteq C$ implies
\begin{align}
\label{eq:TI'}
(A\ominus B)\oplus (B\ominus C)=\emptyset.
\end{align}
\end{thm}

\begin{proof}
To see (\ref{eq:TI}) and (\ref{eq:TI'}), we consider membership functions.
Fix $x\in X$ arbitrarily.
We have
\begin{align*}
\mu_{(A\ominus B)\oplus (B\ominus C)}(x)&=\min\{\mu_{A\ominus B}+\mu_{B\ominus C},\ 1\} \\
&=\min\{\max\{a-b,0\}+\max\{b-c,0\},\ 1\} \\
&=
\begin{cases}
\min\{a-c,1\} & {\rm if}\ a\ge b\ge c, \vspace{1mm}\\
\min\{a-b,1\} & {\rm if}\ a\ge c\ge b, \vspace{1mm}\\
\min\{b-c,1\} & {\rm if}\ b\ge c\ge a, \vspace{1mm}\\
\min\{b-c,1\} & {\rm if}\ b\ge a\ge c, \vspace{1mm}\\
\min\{a-b,1\} & {\rm if}\ c\ge a\ge b 
\end{cases}
\\
&\ge 
\begin{cases}
\min\{a-c,1\} & {\rm if}\ a\ge b\ge c, \vspace{1mm}\\
\min\{a-c,1\} & {\rm if}\ a\ge c\ge b, \vspace{1mm}\\
\min\{a-c,1\} & {\rm if}\ b\ge c\ge a, \vspace{1mm}\\
\min\{a-c,1\} & {\rm if}\ b\ge a\ge c, \vspace{1mm}\\
\min\{a-c,1\} & {\rm if}\ c\ge a\ge b
\end{cases}
\\
&=\mu_{A\ominus C}(x),
\end{align*}
so we have proved (\ref{eq:TI}).
It is obvious from the above that `$A=B$ or $B=C$' is the condition for the equal sign of (\ref{eq:TI}) to hold.
Moreover, if $A\subseteq B\subseteq C$ i.e. $a\le b\le c$, 
\begin{align*}
\mu_{(A\ominus B)\oplus (B\ominus C)}(x)&=\min\{\mu_{A\ominus B}+\mu_{B\ominus C},\ 1\}=\min\{0,1\}=0,
\end{align*}
so we have proved (\ref{eq:TI'}).
Hence, this completes the proof.
\end{proof}

Incidentally, let us check the following negative property.

\begin{prop}
\label{prop:AomBopBomAneO}
Let $A,B,C\in \mathbb F(X)$.
Then, one has
\begin{align}
\label{eq:AomBopBomAneO}
(A\ominus B)\oplus (B\ominus A)\not \equiv \emptyset.
\end{align}
Here $\not \equiv$ stands for the meaning of `It does not always hold.'
\end{prop}

\begin{proof}
To see (\ref{eq:AomBopBomAneO}), we consider membership functions.
Fix $x\in X$ arbitrarily.
We have
\begin{align*}
\mu_{(A\ominus B)\oplus (B\ominus A)}(x)&=\min\{\mu_{A\ominus B}+\mu_{B\ominus A},\ 1\} \\
&=\min\{\max\{a-b,0\}+\max\{b-a,0\},\ 1\} \\
&=
\begin{cases}
\min\{a-b,1\} & {\rm if}\ a\ge b, \vspace{1mm}\\
\min\{b-a,1\} & {\rm if}\ b\ge a
\end{cases}
\\
&\not \equiv 0,
\end{align*}
so this completes the proof.
\end{proof}

\begin{rem}
By virtue of Proposition \ref{prop:AomBopBomAneO}, `$C=A$' cannot be a condition for the equal sign of (\ref{eq:TI}) to hold.
\end{rem}

The arithmetic mean and geometric one are well known and it is also well known that
\begin{align}
\label{eq:agm}
\frac{\alpha+\beta}{2}\ge \sqrt{\alpha\beta};\quad \alpha,\beta\ge 0.
\end{align}
The left (resp. right) hand side of this is called the arithmetic (resp. geometric) mean.
Let us apply these means to a fuzzy set inequality.

\begin{thm}[{\bf Arithmetic and geometric mean for fuzzy sets}]
\label{thm:AGM}
Let $A,B\in \mathbb F(X)$ be satisfied with $0\le \mu_{A\cdot B}(x)\le 1/4$ for any $x\in X$.
Then, one has
\begin{align}
\label{eq:AGM}
\frac{A\oplus B}{2}\supseteq \sqrt{A\cdot B}.
\end{align}
The equality holds if and only if $A=B$.
\end{thm}

\begin{proof}
To see (\ref{eq:AGM}), we consider membership functions.
We have, from (\ref{eq:agm}),
\begin{align*}
\mu_{(A\oplus B)/2}(x)&=\frac{1}{2}\min\{a+b,\ 1\} \\
&=\min\left\{\frac{a+b}{2},\ \frac{1}{2}\right\} \\
&\ge \min\left\{\sqrt{ab},\ \frac{1}{2}\right\} \\
&=\sqrt{ab} \\
&=\sqrt{\mu_{A\cdot B}(x)} \\
&=\mu_{\sqrt{A\cdot B}}(x)
\end{align*}
for any $x\in X$.
Also, it is obvious from the above that `$A=B$' is the condition for the equal sign of (\ref{eq:AGM}) to hold.
Hence, this completes the proof.
\end{proof}

Along with the arithmetic mean and the geometric one, there is the harmonic one on famous means:
\begin{align*}
\frac{2}{\frac{1}{\alpha}+\frac{1}{\beta}}=\frac{2\alpha\beta}{\alpha+\beta},\quad \alpha,\beta\neq 0.
\end{align*}
It is well known that this mean obeys the following inequality:
\begin{align}
\label{eq:ghm}
\sqrt{\alpha\beta}\ge \frac{2\alpha\beta}{\alpha+\beta}.
\end{align}
Let us apply harmonic mean to a fuzzy set inequality by slightly changing the shape as follows.

\begin{thm}[{\bf Geometric and harmonic mean for fuzzy sets}]
\label{thm:GHM}
Let $A,B\in \mathbb F(X)$ be satisfied with $A,B\neq \emptyset$ and $0< \mu_A(x)+\mu_B(x)\le 1$ for any $x\in X$.
Then, one has
\begin{align}
\label{eq:GHM}
\frac{\sqrt{A\cdot B}}{2}\supseteq (A\cdot B)\oslash (A\oplus B).
\end{align}
The equality holds if and only if $A=B$.
\end{thm}

\begin{proof}
To see (\ref{eq:GHM}), we consider membership functions.
Remark that $a+b\ge ab$, since $(a+b)-ab=a(1-b)+b\ge 0$ by virtue of $0\le a,b\le 1$.
We have, from the assumption: $0<a+b\le 1$, 
\begin{align*}
\mu_{(A\cdot B)\oslash (A\oplus B)}(x)&=\min\left\{\frac{ab}{\min\{a+b, 1\}},\ 1\right\} \\
&=\min\left\{\frac{ab}{a+b},\ 1\right\} \\
&=\frac{1}{2}\frac{2ab}{a+b} \\
&\le \frac{1}{2}\sqrt{ab} \\
&=\frac{1}{2}\sqrt{\mu_{A\cdot B}(x)} \\
&=\mu_{(1/2)\sqrt{A\cdot B}}(x)
\end{align*}
for any $x\in X$. 
Also, it is obvious from the above that `$A=B$' is the condition for the equal sign of (\ref{eq:GHM}) to hold.
Hence, this completes the proof.
\end{proof}

\begin{rem}
Theorem \ref{thm:GHM} is formulated as a fuzzy version of the rewritten (crisp) geometric and harmonic mean:
\begin{align*}
\frac{\sqrt{\alpha\beta}}{2}\ge \frac{\alpha\beta}{\alpha+\beta}.
\end{align*}
This is because considering the fuzzy set inequality of (\ref{eq:ghm})-type comes a risk that $\mu_{2(A\cdot B)\oslash (A\oplus B)}(x)\ge 1$.
\end{rem}

\begin{cor}[{\bf AGH(Arithmetic, Geometric and Harmonic) inequality for fuzzy sets}]
Let $A,B\in \mathbb F(X)$ be satisfied with $A,B\neq \emptyset$, $0\le \mu_A(x)\mu_B(x)\le 1/4$ and $0< \mu_A(x)+\mu_B(x)\le 1$ for any $x\in X$.
Then, one has
\begin{align*}
\frac{A\oplus B}{4}\supseteq \frac{\sqrt{A\cdot B}}{2}\supseteq (A\cdot B)\oslash (A\oplus B).
\end{align*}
Both of the equalities hold if and only if $A=B$.
\end{cor}

Recall Cauchy-Schwarz inequality: $|\alpha\gamma+\beta\delta|\le \sqrt{\alpha^2+\beta^2}\sqrt{\gamma^2+\delta^2}$ for any $\alpha,\beta,\gamma,\delta\in \mathbb R$.
Also recall that the equality holds if and only if $\alpha=k\gamma$ and $\beta=k\delta$ for any $k\in \mathbb R$.
Let us see if this inequality holds for fuzzy sets as well.

\begin{thm}[{\bf Cauchy-Schwarz inequality for fuzzy sets}]
\label{thm:CSI}
Let $A,B,C,D\in \mathbb F(X)$ be satisfied with $0\le \mu_A(x)^2+\mu_B(x)^2,\mu_C(x)^2+\mu_D(x)^2\le 1$.
Then, one has
\begin{align}
\label{eq:SineqF}
A\cdot C\oplus B\cdot D\subseteq \sqrt{A^2\oplus B^2}\cdot \sqrt{C^2\oplus D^2}.
\end{align}
The equality holds if and only if $A=\kappa C$ and $B=\kappa D$ and $C=\nu A$ and $D=\nu B$ with $0\le \kappa,\nu\le 1$ satisfying $\kappa\nu=1$.
\end{thm}

\begin{proof}
To see (\ref{eq:SineqF}), we consider membership functions.
We have, from the assumption: $0<a^2+b^2,c^2+d^2\le 1$, 
\begin{align*}
\mu_{A\cdot C\oplus B\cdot D}(x)&=\min\{ac+bd,\ 1\} \\
&\le \min\{\sqrt{a^2+b^2}\sqrt{c^2+d^2},\ 1\} \\
&=\left(\min\{\sqrt{a^2+b^2},\ 1\}\right)\left(\min\{\sqrt{c^2+d^2},\ 1\}\right) \\
&=\sqrt{\min\{a^2+b^2,\ 1\}}\sqrt{\min\{c^2+d^2,\ 1\}} \\
&=\sqrt{\mu_{A^2\oplus B^2}(x)}\sqrt{\mu_{C^2\oplus D^2}(x)} \\
&=\mu_{\sqrt{A^2\oplus B^2}}(x)\mu_{\sqrt{C^2\oplus D^2}}(x) \\
&=\mu_{\sqrt{A^2\oplus B^2}\cdot \sqrt{C^2\oplus D^2}}(x),
\end{align*}
so (\ref{eq:SineqF}) is gained.
Also, it is obvious from the above that `$A=\kappa C$ and $B=\kappa D$' or `$C=\nu A$ and $D=\nu B$' is the condition for the equal sign of (\ref{eq:SineqF}) to hold.
We have to however remark that both of these conditions are essential under $\kappa\nu=1$ (see 2) of Remark \ref{rem:kn=1}), because there is a risk that $C=(1/\kappa)A$ and $D=(1/\kappa)B$ with `$A=\kappa C$ and $B=\kappa D$.'
Here remark that $1/\kappa\ge 1$.
Hence, this completes the proof.
\end{proof}

\begin{rem}
\label{rem:kn=1}
\begin{itemize}
\item[1)] (\ref{eq:SineqF}) is equivalent to 
\begin{align*}
(A\cdot C\oplus B\cdot D)^2\subseteq (A^2\oplus B^2)\cdot (C^2\oplus D^2)
\end{align*}
since $(\min\{ac+bd,1\})^2=\min\{(ac+bd)^2,1\}$.
\item[2)] The reason we assume $\kappa=\nu$ is the following: Since $a=\kappa c$ and $c=\nu a$, we have
\begin{align*}
a=\kappa c=\kappa \nu a,\quad (1-\kappa \nu)a=0,
\end{align*}
so it is necessary that $\kappa \nu=1$.
\end{itemize}
\end{rem}

Bernoulli's inequality is well known as a convenient inequality often used in analysis etc.:
\begin{align*}
(1+x)^m\ge 1+mx\quad (x\ge -1,\ m\in \mathbb N\cup \{0\}).
\end{align*}
We however consider `Generalized Bernoulli's Inequality' to apply this type to fuzzy sets:
\begin{align}
\label{eq:GBineq}
(\alpha+\beta)^m\ge \alpha^m+m\alpha^{m-1}\beta\quad (\alpha,\beta\ge 0,\ m\in \mathbb N).
\end{align}
This is because thinking about $(X\oplus A)^m$ is nonsense as we see in the following theorem.

\begin{thm}[{\bf Generalized Bernoulli's inequality for fuzzy sets}]
Let $m\in \mathbb N$ and $A,B\in \mathbb F(X)$ be satisfied with $0\le \mu_{A}(x)^{m-1}\mu_{B}(x)\le 1/m$.
Then, one has
\begin{align}
\label{eq:GBI}
(A\oplus B)^m\supseteq A^m\oplus m(A^{m-1}\cdot B)
\end{align}
for any $m$.
The equality holds if and only if $A=B=\emptyset$.
\end{thm}

\begin{proof}
To see (\ref{eq:GBI}), we consider membership functions.
Fix $x\in X$ arbitrarily.
Remark that $0\le ma^{m-1}b\le 1$ by the assumption.
We have, from (\ref{eq:GBineq}),
\begin{align*}
\mu_{(A\oplus B)^m}(x)&=(\min\{a+b,\ 1\})^m \\
&=\min\{(a+b)^m,\ 1\} \\
&\ge \min\{a^m+ma^{m-1}b,\ 1\} \\
&=\mu_{A^m\oplus m(A^{m-1}\cdot B)}(x).
\end{align*}
Also, it is obvious from the above that `$A=B=\emptyset$' is the condition for the equal sign of (\ref{eq:GBI}) to hold.
Hence, this completes the proof.
\end{proof}

Chebyshev's inequality is also known:
\begin{align}
\label{eq:Chebyineq}
(\alpha+\beta)(\gamma+\delta)\le 2(\alpha\gamma+\beta\delta)
\end{align}
for any $\alpha,\beta,\gamma,\delta\in \mathbb R$ satisfying $\alpha\ge \beta$ and $\gamma\ge \delta$.
We show the following lemma so as to apply this to fuzzy set inequalities (by slightly changing the shape).

\begin{lem}
\label{lem:ma1mb1}
Let $\alpha,\beta\in \mathbb R$ be satisfied with $\alpha,\beta\ge 0$.
Then, one has
\begin{align*}
\min\{\alpha,1\}\min\{\beta,1\}\le \min\{\alpha\beta,1\}.
\end{align*}
\end{lem}

\begin{proof}
We can set $\alpha\le \beta$ without loss of generality.
\begin{itemize}
\item[i)] If $0\le \alpha\le \beta\le 1$:

${\rm LHS}=\alpha\beta$ and ${\rm RHS}=\alpha\beta$ since $0\le \alpha\beta\le 1$, so ${\rm LHS}={\rm RHS}$.
\item[ii)] If $0\le \alpha\le 1\le \beta$:

Since $\alpha\beta-\alpha=\alpha(\beta-1)\ge 0$, ${\rm LHS}=\min\{\alpha,1\}\le \min\{\alpha\beta,1\}={\rm RHS}$.
\item[iii)] If $1\le \alpha\le \beta$:

${\rm LHS}=1$ and ${\rm RHS}=1$ since $\alpha\beta\ge 1$, so ${\rm LHS}={\rm RHS}$.
\end{itemize}
Hence, we have found that the desired equation holds in any case.
\end{proof}

\begin{thm}[{\bf Pseudo Chebyshev's inequality for fuzzy sets}]
\label{thm:PCheby}
Let $A,B,C,D\in \mathbb F(X)$ be satisfied with $A\supseteq B$, $C\supseteq D$ and $0\le \mu_A(x)+\mu_B(x),\mu_C(x)+\mu_D(x)\le 1$.
Then, one has
\begin{align}
\label{eq:Cheby}
\frac{(A\oplus B)\cdot (C\oplus D)}{2}\subsetneq (A\cdot C)\oplus (B\cdot D).
\end{align}
\end{thm}

\begin{proof}
To see (\ref{eq:Cheby}), we consider membership functions.
Fix $x\in X$ arbitrarily.
Since $a\ge b,c\ge d$ and $0\le a+b,c+d\le 1$ from the assumption, (\ref{eq:Chebyineq}) and Lemma \ref{lem:ma1mb1} imply that
\begin{align*}
\mu_{(A\oplus B)\cdot (C\oplus D)/2}(x)&=\frac{1}{2}\min\{a+b,\ 1\}\min\{c+d,\ 1\} \\
&\le \frac{1}{2}\min\{(a+b)(c+d),\ 1\} \\
&\le \frac{1}{2}\min\{2(ac+bd),\ 1\} \\
&<\frac{1}{2}\min\{2(ac+bd),\ 2\} \\
&=\min\{ac+bd,\ 1\} \\
&=\mu_{(A\cdot C)\oplus (B\cdot D)}(x).
\end{align*}
Hence, this completes the proof.
\end{proof}

\begin{rem}
Theorem \ref{thm:PCheby} is formulated as a fuzzy version of the rewritten (\ref{eq:Chebyineq}):
\begin{align*}
\frac{(\alpha+\beta)(\gamma+\delta)}{2}\le \alpha\gamma+\beta\delta.
\end{align*}
This is because considering the fuzzy set inequality of (\ref{eq:Chebyineq})-type comes a risk that $\mu_{2(A\cdot C+B\cdot D)}(x)\ge 1$.
\end{rem}

We recall the following simple but important inequalities.
It may not be necessary, but we will give the proofs.

\begin{lem}
\label{lem:2pmap+bp}
Let $\alpha,\beta,p\ge 0$.
Then, one has
\begin{align}
\label{eq:2pmapbp}
(\alpha+\beta)^p\le 2^p\max\{\alpha^p, \beta^p\}.
\end{align}
The equality holds if and only if $\alpha=\beta$.
In particular, one has
\begin{align}
\label{eq:2pap+bp}
(\alpha+\beta)^p\le 2^p(\alpha^p+\beta^p).
\end{align}
The equality holds if and only if $\alpha=\beta=0$.
\end{lem}

\begin{proof}
$(\alpha+\beta)^p\le (2\max\{\alpha,\beta\})^p=2^p\max\{\alpha^p,\beta^p\}$.
It is obvious from the above that `$\alpha=\beta$' is the condition for the equal sign of (\ref{eq:2pmapbp}) to hold.
If you proceed with the estimate further, we have $(\alpha+\beta)^p\le 2^p\max\{\alpha^p,\beta^p\}\le 2^p(\alpha^p+\beta^p)$.
It is also obvious from the above that `$\alpha=\beta=0$' is the condition for the equal sign of (\ref{eq:2pap+bp}) to hold.
Hence, this completes the proof.
\end{proof}

(\ref{eq:2pap+bp}) is often used in analysis.
For instance, it can be shown that the $p$-Lebesgue space $L^p$, $1\le p\le +\infty$, is a vector space by using (\ref{eq:2pap+bp}):
If $f,g\in L^p$, then
\begin{align*}
\int|f+g|^p=2^p\left(\int|f|^p+\int|g|^p\right)<+\infty.
\end{align*}
Thus $f+g\in L^p$ (see \cite{F} and so on for details).

Let us obtain a fuzzy set version of (\ref{eq:2pap+bp}) by slightly changing the shape as follows.

\begin{thm}
\label{thm:AoBp2pApcBp}
Let $A,B\in \mathbb F(X)$.
Then, one has
\begin{align}
\label{eq:AoBp2pApcBp}
\frac{(A\oplus B)^p}{2^p}\subseteq A^p\cup B^p
\end{align}
for any $0\le p<1$.
The equality holds if and only if $A=B$ and $p=0$.
\end{thm}

\begin{proof}
To see (\ref{eq:AoBp2pApcBp}), we consider membership functions.
Fix $x\in X$ arbitrarily.
We have, from (\ref{eq:2pmapbp}),
\begin{align*}
\mu_{(A\oplus B)^p/2^p}(x)&=\frac{1}{2^p}(\min\{a+b,\ 1\})^p \\
&=\frac{1}{2^p}\min\{(a+b)^p,\ 1\} \\
&\le \frac{1}{2^p}\min\{2^p\max\{a^p,b^p\},\ 1\} \\
&\le \frac{1}{2^p}\min\{2^p\max\{a^p,b^p\},\ 2^p\} \\
&=\min\{\max\{a^p,b^p\},\ 1\} \\
&=\max\{a^p,b^p\} \\
&=\mu_{A^p\cup B^p}(x).
\end{align*}
Also, it is obvious from the above and Lemma \ref{lem:2pmap+bp} that `$A=B$ and $p=0$' is the condition for the equal sign of (\ref{eq:AoBp2pApcBp}) to hold.
Hence, this completes the proof.
\end{proof}

\begin{rem}
Theorem \ref{thm:AoBp2pApcBp} is formulated as a fuzzy version of the rewritten (\ref{eq:2pmapbp}):
\begin{align*}
\frac{(\alpha+\beta)^p}{2^p}\le \max\{\alpha^p,\beta^p\}.
\end{align*}
This is because considering the fuzzy set inequality of (\ref{eq:2pmapbp})-type comes a risk that $\mu_{2^p(A^p\cup B^p)}(x)\ge 1$.
\end{rem}

\section{Conclusion}
Inclusion relations for fuzzy sets correspond to inequalities for the membership functions.
That is, we will be able to call inclusion relations for fuzzy sets Fuzzy Set Inequalities.
We found that it is rare that conventional inequalities hold true for fuzzy sets, without any assumptions, in the sense of inclusion (e.g. `$0\le \mu_{A\cdot B}(x)\le 1/4$' in Theorem \ref{thm:AGM}, `$0<\mu_A(x)+\mu_B(x)\le 1$' in Theorem \ref{thm:GHM}, `$0\le \mu_A(x)^2+\mu_B(x)^2,\mu_C(x)^2+\mu_D(x)^2\le 1$' in Theorem \ref{thm:CSI} and so on).
Moreover, the assumptions can be strong conditions.
However, that may be improved by replacing them with other operations.
We would like to make that investigation a future topic.

Readers interested in other operations and basic fuzzy set inequalities should refer to e.g. \cite{MT,Mp}.
There are so many formulae on various operations in \cite{MT,Mp}.
Furthermore, readers who want to know various inequalities, from basic to maniac, are recommended to refer to the website \cite{J}.

Fuzzy theory was originally born as a field of applied mathematics, but the study of `pure fuzzy mathematics' has also been actively conducted.
The results can be seen in e.g. the recently published \cite{SG}.

By the way, we derived some fundamental fuzzy set inequalities in the present note, but we did not touch on those applications.
We hope that the applications of fuzzy set inequalities obtained in the present note will be found in other fields.

{\small
\noindent
Author: {\sc Norihiro Someyama}

He received a M.Sc. degree from Gakushuin University in 2014 and completed the Ph.D program without a Ph.D. degree the same university in 2017.
He is a head priest of Shin-yo-ji Buddhist Temple in Japan.
His research interests are the spectral theory of Schr\"{o}dinger operators and the theory of fuzzy Schr\"{o}dinger equations.
He received the Member Encouragement Award of Biomedical Fuzzy System Association for his lecture entitled `{\it Characteristic Analysis of Fuzzy Graph and its Application IV}' in November 2018 and the Excellent Presentation Award of National Congress of Theoretical and Applied Mechanics / JSCE Applied Mechanics Symposium for his lecture entitled `{\it Number of Eigenvalues of Non-self-adjoint Schr\"{o}dinger Operators with Dilation Analytic Complex Potentials}' in October 2019.
}

\end{document}